\documentclass{amsart}
\usepackage{amssymb,amsmath,amsthm}
\usepackage[foot]{amsaddr}
\title{Path categories and propositional identity types}
\date{\today}
\author{Benno van den Berg}
\usepackage{ast2}
\usepackage{bussproofs}
\usepackage[margin=1.7in]{geometry}

\begin{document}

\maketitle

\begin{abstract}
Connections between homotopy theory and type theory have recently attracted a lot of attention, with Voevodsky's univalent foundations \cite{voevodsky11} and the interpretation of Martin-L\"of's identity types in Quillen model categories \cite{awodeywarren09} as some of the highlights. In this paper we establish a connection between a natural weakening of Martin-L\"of's rules for the identity types which has been considered by Cohen, Coquand, Huber and M\"ortberg in their work on a constructive interpretation of the univalence axiom \cite{cohenetal15} on the one hand, and the notion of a path category, a slight variation on the classic notion of a category of fibrant objects due to Brown \cite{brown73}, on the other. This involves showing that the syntactic category associated to a type theory with weak identity types carries the structure of a path category, strengthening earlier results by Avigad, Lumsdaine and Kapulkin \cite{avigadetal15}. In this way we not only relate a well-known concept in homotopy theory with a natural concept in logic, but also provide a framework for further developments.
\end{abstract}

\section{Introduction}

Martin-L\"of's rules for the identity types have led to several correspondences between notions from type theory and logic on the one hand and notions from homotopy theory and category theory on the other. The aim of this paper is to establish another correspondence between a well-known categorical concept in homotopy theory and a natural weakening of Martin-L\"of's rules for the identity type.

At first blush, there is no reason to expect such connections; indeed, the ideas that guided Martin-L\"of in setting up the rules for the identity types were more philosophical in character and, if anything, point in the opposite direction. His idea was that for any type $A$ and any pair of elements $a, b \in A$ there should be, besides the judgement $a = b \in A$, expressing that $a$ and $b$ are \emph{definitionally equal}, a type ${\rm Id}_A(a, b)$ whose elements are proofs of the equality of $a$ and $b$. This leads to a second, and weaker, notion of equality, defined by saying that $a$ and $b$ are \emph{propositionally equal} if there is a term $p \in {\rm Id}_A(a, b)$. The rules for the identity types have the form of an inductive definition, with elements of identity types generated inductively from reflexivity terms ${\bf r}(a) \in {\rm Id}_A(a,a)$,  witnessing the equality of $a$ with itself. Given this starting point, it was natural to expect that all elements in an identity type should be provably equal to a reflexivity term, and an identity type ${\rm Id}_A(a, b)$ could only be inhabited if $a = b \in A$ is derivable as well.

However, these ideas were refuted in a seminal paper by Hofmann and Streicher \cite{hofmannstreicher98}. In this paper, Hofmann and Streicher make two technical contributions: first, they show that the identity types not only determine an equivalence relation on every type, but give it the structure of a groupoid as well. More precisely, because equality is provably an equivalence relation, there is for any $p \in {\rm Id}_A(a, b)$ an element ${\bf s}p \in {\rm Id}_A(b, a)$ and for any pair of elements $p \in {\rm Id}_A(a, b)$ and $q \in {\rm Id}_A(b, c)$ an element ${\bf t}(p, q) \in {\rm Id}_A(a, c)$; in addition, there are the  reflexivity terms ${\bf r}a \in {\rm Id}_A(a, a)$. This much could be expected from any proof-relevant treatment of equality; however, Hofmann and Streicher show that in type theory these operations ${\bf r}$, ${\bf s}$ and ${\bf t}$ give $A$ the structure of a groupoid, at least up to elements in the iterated identity types of the form ${\rm Id}_{{\rm Id}_A(a, b)}(p, q)$.

The second, and far more involved, contribution of Hofmann and Streicher is the construction of a model of type theory in which the types are interpreted as groupoids; the idea is that if a groupoid interprets some type $A$, then the objects in this groupoid interpret elements $a, b \in A$ and the discrete groupoid on the set of arrows between these objects interprets ${\rm Id}_A(a, b)$. Since there can be distinct parallel arrows in a groupoid, this model shows the impossibility of proving that any two elements of  ${\rm Id}_A(a, b)$ must be propositionally equal.

Given these contributions, the connection to homotopy theory and category theory starts to look compelling, if not inevitable. Indeed, the properties of the identity type uncovered by Hofmann and Streicher make sense if we understand types as spaces. On this analogy one reads elements of a type as points in a space, ${\rm Id}_A(a, b)$ as the space of paths between $a$ and $b$ and elements in iterated identity types as homotopies between paths. In a topological space paths can be composed and reversed, satisfying the laws of a groupoid up to homotopy, while the constant paths act as identities, just as in type theory.

But if this analogy is to be trusted, much more should be true. Indeed, a type with all its iterated identity types should have the structure of an $\infty$-groupoid: this was already conjectured by Hofmann and Streicher and subsequently shown to be true in \cite{vandenberggarner11,lumsdaine10}. Also, the category of all $\infty$-groupoids ought to yield a model of type theory too; Voevodsky showed that this is true as well \cite{voevodsky11,kapulkinetal12}. In this way we obtain a precise connection between type theory and homotopy theory, in that the axioms of an $\infty$-groupoid capture precisely the algebraic structure both of a type with its iterated identity types and a space and its iterated path spaces. The idea to read types in type theory as homotopy types of spaces has since led to a lot of new developments; we refer to \cite{UFP13}.

Here we start from the connection between the identity types and weak factorisation systems, another notion from homotopy theory. In abstract homotopy theory such weak factorisation systems abound; indeed, Quillen model structures, which are highly influential as abstract environments in which one can do homotopy theory, are categories equipped with two interlocking weak factorisation systems \cite{quillen67,hirschhorn03,hovey99}. The rough idea is that such weak factorisation systems provide a sound and complete semantics for the identity types, in that the classifying category associated to any type theory with identity types comes equipped with a weak factorisation system, while the rules for the identity types can be interpreted in any category with a weak factorisation system. The former was shown to be true by Gambino and Garner \cite{gambinogarner08}, but the latter is true only with certain qualifications. Indeed, weak factorisation systems only yield ``non-split'' models, the reason being that the structure of a weak factorisation system is not sufficiently rigid to provide interpretations of the identity types which work well with substitution (this is known as the ``coherence problem''). At present it is unclear whether there exists a general method for splitting such models in a way which gives us genuine models of the identity type, although we do possess methods which take care of most of the concrete cases of interest (for more on this, we refer to \cite{lumsdainewarren15}). So the correct statement would be that categories with weak factorisation systems are non-split models of the identity types and that to obtain a model of the identity types one needs something like a homotopy-theoretic model of the identity types as in  \cite{vandenberggarner12}.

In this paper we establish a similar kind of correspondence, where on the homotopy-theoretic side we have the notion of a \emph{path category}, short for a category with path objects. The notion of a path category is a slight strengthening of Brown's classic notion of a category of fibrant objects \cite{brown73} and was introduced in \cite{bergmoerdijk16}, where also many of its basic properties were established. Like Quillen's notion of a model category such categories provide abstract settings in which the basic concepts and results from homotopy theory can be interpreted. However, path categories differ from Quillen model categories in that they are based on two classes of maps, called weak equivalences and fibrations, and there is no third class of maps called cofibrations.

Another difference is that path categories carry no underlying weak factorisation system: what one does have in a path category is that in any commuting square with a weak equivalence on the left and a fibration on the right there is a diagonal filler which makes the resulting lower triangle commutative, whilst making the upper triangle commute up to (fibrewise) homotopy (this was shown in \cite{bergmoerdijk16}). For the interpretation of the identity types in path categories this means that one cannot soundly interpret the usual computation rule for the identity type; however, there is a natural weakening of this rule which can be interpreted. Indeed, the computation rule states a definitional equality between two terms; if one replaces this by a propositional equality, then this weaker rule can be interpreted. We refer to the identity type with this weakened computation rule as the \emph{propositional identity type}.

These propositional identity types have recently been considered by Cohen, Coquand, Huber and M\"ortberg in their work on a constructive interpretation of the univalence axiom \cite{cohenetal15}. In their work they build a model of type theory with Voevodsky's univalence axiom inside a constructive metatheory. However, they do not succeed in interpreting the usual rules for the identity types: for their ``path types'' the computation rule holds only in a propositional form. At present, it is unclear whether a constructive interpretation of a type theory with both the usual rules for the identity types and univalence can be found.

In this paper we establish a precise relation between path categories and propositional identity types. On the one hand, path categories allow for a non-split interpretation of the propositional identity types; on the other hand, the syntactic category associated to any type theory with propositional identity types carries the structure of a path category. The latter strengthens Theorem 3.2.5 and Lemma 3.2.14 in \cite{avigadetal15}, where it was shown that the syntactic category associated to full Martin-L\"of type theory with the usual rules for the identity types has this structure; the main result here is that we show that a basic type theory equipped only with propositional identity types suffices for this purpose.

The precise contents of this paper are as follows. In Section 2 we recall the main features of the syntax of type theory and establish our notational conventions. We borrow the notion of a tribe from Joyal which gives us a basic semantics for type theory. In Section 3 we discuss propositional identity types and establish some categorical properties of the syntactic category associated to any type theory equipped with such propositional identity types. In Section 4 we recall the notion of a path category from \cite{bergmoerdijk16} and discuss how these provide non-split models for propositional identity types. In Sections 5 and 6 we establish that the syntactic category associated to any type theory with propositional identity types is a path category; in Section 5 we prove this under the additional assumption that the type theory comes equipped with strong unit and sum types, leaving a proof of the general case to Section 6. The paper ends with an appendix containing technical results that are needed at various points in the paper.

The research reported here was done whilst the author was a visiting fellow at the Isaac Newton Institute for Mathematical Sciences in the programme ``Mathematical, Foundational and Computational Aspects of the Higher Infinite (HIF)"  funded by EPSRC grant EP/K032208/1. The author thanks the Institute for excellent working conditions, which should in no way be held responsible for the slow pace at which this paper was written.
\section{Syntax and semantics}

For the convenience of the reader we recall here some basic facts about the syntax of dependent type theory; we also establish some notational and terminological conventions that will be used throughout this paper. All this material is absolutely standard and for more comprehensive treatments the reader could consult \cite{martinlof84,nordstrometal90,hofmann97,jacobs99,pitts00}.

\subsection{General remarks about syntax} Type theory is a formal system for deriving statements of the form
\[ \mathcal{J} \, [ \, \Gamma \, ], \]
where $\mathcal{J}$ is a \emph{judgement} and $\Gamma$ is a \emph{context}. Judgements can have one of the following four forms:
\[ A \in \mbox{Type} \qquad a \in A \qquad A = B \in \mbox{Type} \qquad a = b \in A. \]
The meaning of first statement is that $A$ is a well-formed expression denoting a type, the second means that $a$ is a well-formed expression denoting an object of type $A$, while the third statement means that $A$ and $B$ are definitionally equal type expressions, and the fourth means that $a$ and $b$ are definitionally equal expressions for objects of type $A$.

Any judgement is always made in a \emph{context}. The purpose of the context is to make explicit the types of all the free variables in the judgement. Indeed, a context is of the form
\[ \Gamma = \, [ \, x_0 \in A_0, x_1 \in A_1(x_0), \ldots, x_n \in A_n(x_0,\ldots,x_{n-1}) \, ], \]
where $x_0,\ldots,x_{n-1}$ are distinct variables, and the only variables which may occur freely in $A_i$ are $x_0,\ldots,x_{i-1}$, as indicated. The system will be built in such a way that for a $\Gamma$ as above, the statement $\mathcal{J} \, [\Gamma]$ is derivable only if
\begin{displaymath}
\begin{array}{l}
 A_0 \in {\rm Type} \, [] \\ A_1(x_0) \in {\rm Type} \, [x_0 \in A_0] \\ \ldots \\ A_n(x_0,\ldots,x_{n-1}) \in {\rm Type} \, [ x_0 \in A_0, \ldots, x_{n-1} \in A_{n-1}(x_0,\ldots, x_{n-2}) ]
\end{array}
\end{displaymath}
are derivable as well.

The rules in type theory are of the form
\begin{prooftree}
\AxiomC{$\mathcal{J}_1 \, [ \, \Gamma_1 \, ]$}
\AxiomC{$\mathcal{J}_2 \, [ \, \Gamma_2 \, ]$}
\AxiomC{$\ldots$}
\AxiomC{$\mathcal{J}_n \, [ \, \Gamma_n \, ]$}
\QuaternaryInfC{ $\mathcal{J} \, [ \, \Gamma \, ]$}
\end{prooftree}
meaning that once each of the $\mathcal{J}_i \, [ \, \Gamma_i \, ]$ has been derived, one may apply the rule to derive $\mathcal{J} \, [ \, \Gamma \, ]$ as well. In case $n = 0$ the rule is an axiom and says that $\mathcal{J} \, [ \, \Gamma \, ]$ is derivable without any assumptions. All the rules that we will see will have the form
\begin{prooftree}
\AxiomC{$\mathcal{J}_1 \, [ \, \Delta, \Gamma_1 \, ]$}
\AxiomC{$\mathcal{J}_2 \, [ \, \Delta, \Gamma_2 \, ]$}
\AxiomC{$\ldots$}
\AxiomC{$\mathcal{J}_n \, [ \, \Delta, \Gamma_n \, ]$}
\QuaternaryInfC{ $\mathcal{J} \, [ \, \Delta, \Gamma \, ]$}
\end{prooftree}
where there is some context $\Delta$ shared by both the assumptions and the conclusion. Most of the time such shared contexts $\Delta$ are left implicit when writing down rules. For example, one of the axioms of type theory will be written as
\begin{prooftree}
\AxiomC{$A \in {\rm Type}$}
\UnaryInfC{ $A = A$}
\end{prooftree}
but this should really be understood as
\begin{prooftree}
\AxiomC{$A \in {\rm Type} \, [ \, \Gamma \, ]$}
\UnaryInfC{ $A = A \, [ \, \Gamma \, ]$}
\end{prooftree}
for arbitrary contexts $\Gamma$. Also the weakening rule
\begin{prooftree}
\AxiomC{$\mathcal{J} \, [ \, \Gamma \, ]$}
\AxiomC{$A \in \mbox{Type} $}
\BinaryInfC{$\mathcal{J} \, [ \, x \in A, \Gamma \, ]$}
\end{prooftree}
below should be understood as
\begin{prooftree}
\AxiomC{$\mathcal{J} \, [ \, \Delta, \Gamma \, ]$}
\AxiomC{$A \in \mbox{Type} \, [ \, \Delta \, ]$}
\BinaryInfC{$\mathcal{J} \, [ \, \Delta, x \in A, \Gamma \, ]$}
\end{prooftree}
This is the only rule with a side-condition: here $x$ should not occur in $\Gamma$ or $\Delta$.

With this in mind, the basic rules of type theory are the following:
\begin{itemize}
\item[--] Axiom, weakening and substitution:
\begin{center}
\begin{tabular}{c} \\
\AxiomC{$$}
\UnaryInfC{ $x \in A \, [ \, x: A \, ]$}
\DisplayProof \\ \\
\end{tabular}
\begin{tabular}{cc}
\AxiomC{$\mathcal{J} \, [ \, \Gamma \, ]$}
\AxiomC{$A \in \mbox{Type} $}
\BinaryInfC{$\mathcal{J} \, [ \, x \in A, \Gamma \, ]$}
\DisplayProof &
\AxiomC{$\mathcal{J} \, [ \, x \in A, \Gamma \, ]$}
\AxiomC{$a \in A $}
\BinaryInfC{$\mathcal{J}[a/x] \, [ \, \Gamma[a/x] \, ]$}
\DisplayProof \\ \\
\end{tabular}
\end{center}
\item[--] Reflexivity, symmetry, and transitivity of definitional equality of types
\begin{center}
\begin{tabular}{ccc} \\
\AxiomC{$A \in \mbox{Type}$}
\UnaryInfC{ $A = A$ }
\DisplayProof &
\AxiomC{$A = B $}
\UnaryInfC{ $B = A $ }
\DisplayProof &
\AxiomC{$A = B$}
\AxiomC{$B = C'$}
\BinaryInfC{ $A = C$ }
\DisplayProof \\ \\
\end{tabular}
\end{center}
\item[--] Reflexivity, symmetry, and transitivity of definitional equality of terms
\begin{center}
\begin{tabular}{ccc} \\
\AxiomC{$a \in A$}
\UnaryInfC{ $a = a \in A$ }
\DisplayProof &
\AxiomC{$a =b \in A$}
\UnaryInfC{ $b = a \in A$ }
\DisplayProof &
\AxiomC{$a =b \in A$}
\AxiomC{$b = c \in A$}
\BinaryInfC{ $a = c \in A$ }
\DisplayProof \\ \\
\end{tabular}
\end{center}
\item[--] Compatibility rules for definitional equality
\begin{center}
\begin{tabular}{cc} \\
\AxiomC{$a \in A$}
\AxiomC{$A = B$}
\BinaryInfC{ $a \in B$ }
\DisplayProof &
\AxiomC{$a = b \in A$}
\AxiomC{$A = B$}
\BinaryInfC{ $a = b \in B$ }
\DisplayProof \\ \\
\end{tabular}
\end{center}
\end{itemize}

\subsection{Classifying category} To any dependent type theory with the rules above we can associate a category, which we will call the \emph{classifying} or \emph{syntactic category}. The objects of this category are equivalence classes of contexts
\[ \Gamma = \, [ \, x_0 \in A_0, x_1 \in A_1(x_0), \ldots, x_n \in A_n(x_0,\ldots,x_{n-1}) \, ], \]
of the same length, where we identify $\Gamma$ with a context
\[ \Gamma' = \, [ \, y_0 \in B_0, y_1 \in B_1(y_0), \ldots, y_n \in B_n(y_0,\ldots,y_{n-1}) \, ] \]
if the following statements are derivable in the type theory:
\begin{displaymath}
\begin{array}{l}
 A_0 = B_0 \in {\rm Type} \, [] \\ A_1(x_0) = B_1(x_0) \in {\rm Type} \, [x_0 \in A_0] \\ \ldots \\ A_n(x_0,\ldots,x_{n-1})= B_n(x_0, \ldots, x_{n-1}) \in {\rm Type} \, [ x_0 \in A_0, \ldots, x_{n-1} \in A_{n-1}(x_0,\ldots, x_{n-2}) ]
\end{array}
\end{displaymath}
A morphism $f: \Delta \to \Gamma$, where $\Gamma$ is as above, is an equivalence class of terms $(t_0, \ldots, t_{n})$ for which the following statements are derivable:
\begin{displaymath}
\begin{array}{l}
t_0 \in A_0 \in {\rm Type} \, [\, \Delta \,] \\ t_1 \in A_1(t_0) \in {\rm Type} \, [ \, \Delta ] \\ \ldots \\ t_n \in A_n(t_0,\ldots,t_{n-1}) \, [ \, \Delta \, ]
\end{array}
\end{displaymath}
and where we identify $(t_0, \ldots, t_n)$ with $(s_0,\ldots,s_n)$ if the following statements are derivable:
\begin{displaymath}
\begin{array}{l}
s_0 = t_0 \in A_0 \in {\rm Type} \, [\, \Delta \,] \\s _ 1 = t_1 \in A_1(t_0) \in {\rm Type} \, [ \, \Delta ] \\ \ldots \\ s_n = t_n \in A_n(t_0,\ldots,t_{n-1}) \, [ \, \Delta \, ]
\end{array}
\end{displaymath}
The verification that this defines a category with composition given by making suitable substitutions can be found in the sources mentioned at the beginning of this section.

There are several classes of morphisms in this category which are of interest. First of all, there are the \emph{display maps} which are maps of the form $[\Gamma, x \in A] \to [\Gamma]$ dropping the last type from the context (more precisely, if $\Gamma$ is as above this is the equivalence class of the sequence $(x_0,\ldots,x_n)$). If we close these maps under identities and composition, we get the morphisms of the form $[\Gamma, \Delta] \to \Gamma$ dropping a final segment from a context: these maps are often called \emph{dependent projections}. If we also close under isomorphism, we obtain what we will call the \emph{fibrations}: that is, fibrations are morphisms which are isomorphic to dependent projecions.

In the present context, the fibrations are the most important class, and the most important property of these fibrations is that they are closed under pullbacks. Indeed, if $f = [t_0,\ldots,t_n]: \Delta \to \Gamma$ is as above and $[\Gamma, x \in A] \to [\Gamma]$ is a display map dropping the final type $A = A(x_0,\ldots,x_n)$, then
\diag{ [\Delta, y \in A(t_0,\ldots,t_n) ] \ar[d] \ar[r] & [\Gamma, x \in A] \ar[d] \\
\Delta \ar[r] & \Gamma }
is a pullback. So it follows from pullback pasting that if $f: \Delta \to \Gamma$ is an arbitrary map and $p: \Gamma' \to \Gamma$ is a fibration, then the pullback of $p$ along $f$ exists and is a fibration as well. Furthermore, it is easy to see that identity morphisms are fibrations, fibrations are closed under composition and the empty context $[]$ is a terminal object in the classifying category with the unique map $!: \Gamma \to []$ always being a fibration.

\subsection{Type theories with strong sums} In the sequel it will often be convenient to assume that our type theories have strong sums; indeed, we will first obtain our main results in the setting of type theories with strong sums and then we will eliminate this assumption. The main reason why the assumption of strong sums is so convenient is that in the classifying category of any type theory with strong sums every fibration is isomorphic to a display map. (The following discussion should be compared to Exercise 10.1.9 on page 593 of \cite{jacobs99}).

We will say that a type theory has \emph{strong sums} if it contains a type 1 with rules
\begin{center}
\begin{tabular}{cc} \\
\AxiomC{}
\UnaryInfC{ $1 \in \mbox{Type}$ }
\DisplayProof &
\AxiomC{}
\UnaryInfC{ $* \in 1$ }
\DisplayProof
\\ \\
\end{tabular}
\end{center}
and definitional equality \[ a = * \in 1, \]
as well as a type constructor $\Sigma$ with rules
\begin{center}
\begin{tabular}{cc} \\
\AxiomC{$A \in \mbox{Type}$}
\AxiomC{$B \in \mbox{Type} \, [x:A]$}
\BinaryInfC{ $\Sigma x \in A.B \in \mbox{Type}$ }
\DisplayProof &
\AxiomC{$a \in A$}
\AxiomC{$b \in B[a/x]$}
\BinaryInfC{ ${\bf p}ab \in \Sigma x \in A.B$ }
\DisplayProof \\ \\
\end{tabular}
\begin{tabular}{cc}
\AxiomC{$c \in \Sigma x \in A.B$}
\UnaryInfC{ ${\bf p}_0c \in A$ }
\DisplayProof &
\AxiomC{ $c \in \Sigma x \in A.B$ }
\UnaryInfC{ ${\bf p}_1c \in B[{\bf p}_0c/x]$ }
\DisplayProof \\ \\
\end{tabular}
\end{center}
and definitional equalities
\[ {\bf p}_0 ({\bf p}ab) = a \in A, \qquad {\bf p}_1({\bf p}ab) = b \in B[a/x], \qquad {\bf p}({\bf p}_0c)({\bf p}_1c) = c \in \Sigma x \in A.B.\]

\begin{rema}{butwhataboutcongruence} We will follow the usual type-theoretic convention in leaving the congruence rules for all the type and term constructors implicit. For the strong $\Sigma$-type this means that we also have the following rules:
\begin{center}
\begin{tabular}{cc} \\
\AxiomC{$A = A' \in \mbox{Type}$}
\AxiomC{$B = B' \in \mbox{Type} \, [x:A]$}
\BinaryInfC{ $\Sigma x \in A.B = \Sigma x \in A'. B' \in \mbox{Type}$ }
\DisplayProof &
\AxiomC{$a = a' \in A$}
\AxiomC{$b = b' \in B[a/x]$}
\BinaryInfC{ ${\bf p}ab = {\bf p}a'b' \in \Sigma x \in A.B$ }
\DisplayProof \\ \\
\end{tabular}
\begin{tabular}{cc}
\AxiomC{$c = c' \in \Sigma x \in A.B$}
\UnaryInfC{ ${\bf p}_0c = {\bf p}_0c' \in A$ }
\DisplayProof &
\AxiomC{ $c = c' \in \Sigma x \in A.B$ }
\UnaryInfC{ ${\bf p}_1c = {\bf p}_1 c' \in B[{\bf p}_0c/x]$ }
\DisplayProof \\ \\
\end{tabular}
\end{center}
We will assume that for all type and term constructors we have similar congruence rules.
\end{rema}

\begin{prop}{fibranddisplaymapsforttwithstrongsums}
In the classifying category of a type theory with strong sums every fibration is isomorphic to a display map.
\end{prop}
\begin{proof}
Let $[\Gamma, \Delta] \to \Gamma$ be a dependent projection in the classifying category of a type theory with strong sums. It is not hard to see that this map is isomorphic to $[\Gamma, x \in \Sigma \Delta] \to \Gamma$, where $\Sigma \Delta$ is the type in context $\Gamma$ defined by induction on the length of $\Delta$, as follows:
\begin{eqnarray*}
\Sigma [] & = & 1, \\
\Sigma [\Delta, x \in A] & = & \Sigma y \in \Sigma \Delta. A({\bf q}_0y, \ldots, {\bf q}_{n-1}y).
\end{eqnarray*}
with ${\bf q}_{i} = {\bf p}_1{\bf p}_0^{n - 1- i}$. It follows that every fibration is isomorphic to a display map.
\end{proof}

\subsection{Tribes} Abstracting away from the concrete details of the syntactic category we arrive at:
\begin{defi}{displaymapcategory} A \emph{tribe} (Joyal) is a category \ct{C} with a terminal object 1 in which we have selected a class of morphism called the \emph{fibrations}, satisfying the following axioms:
\begin{enumerate}
\item Isomorphisms are fibrations.
\item Fibrations are closed under composition.
\item For any object $X$ the unique arrow $X \to 1$ is always a fibration.
\item If $p: X' \to X$ is a fibration and $f: Y \to X$ is arbitrary, then there is a pullback square
\diag{ Y' \ar[d]_{q} \ar[r] & X' \ar[d]^p \\
Y \ar[r]_f & X }
in which $q$ is a fibration as well.
\end{enumerate}
\end{defi}

If \ct{C} is a tribe and $X$ is an arbitrary object in \ct{C}, then we can consider the full subcategory of $\ct{C}/X$ whose objects are fibrations with codomain $X$. This category, which we will denote by $\ct{C}(X)$, again has the structure of a tribe if we declare a map to be a fibration in $\ct{C}(X)$ precisely when its underlying morphism in \ct{C} is a fibration there. In fact, we have the following proposition.

\begin{prop}{slicingfordisplaymapcats}
If \ct{C} is a tribe and $X$ is an object in \ct{C}, then $\ct{C}(X)$ is again a tribe. Moreover, if $f: Y \to X$ is an arbitrary morphism in \ct{C}, then pulling back along $f$ determines a functor
\[ f^*: \ct{C}(X) \to \ct{C}(Y), \]
called \emph{change of base}, which preserves the tribal structure (that is, it preserves the terminal object, fibrations, as well as pullbacks of fibrations along arbitrary maps). If $f$ is a fibration, then this functor $f^*$ has a left adjoint $\Sigma_f$ given by postcomposition with $f$.
\end{prop}

\subsection{Equivalence relations} In this paper equivalence relations play an important r\^ole. They can be defined in the general context of a tribe, as follows.

\begin{defi}{eqrelation} Let $x: X \to I$ be a fibration in a tribe \ct{C}. An \emph{equivalence relation on $x: X \to I$} is a fibration $p = (p_1,p_2): R \to X \times_I X$ for which there are:
\begin{enumerate}
\item a morphism $\rho: X \to R$ (witnessing reflexivity) such that $p\rho = \Delta_X$, and
\item a morphism $\sigma: R \to R$ (witnessing symmetry) such that $p\sigma = (p_2,p_1)$, and
\item a morphism $\tau: R \times_X R \to R$ (witnessing transitivity) such that \[ p \tau = (p_1\pi_1, p_2\pi_2), \] where $R \times_X R, \pi_1$ and $\pi_2$ refer to the pullback
\diag{ R \times_X R \ar[r]^(.6){\pi_2} \ar[d]_{\pi_1} & R \ar[d]^{p_1} \\
R \ar[r]_{p_2} & X. }
\end{enumerate}
\end{defi}

Suppose $x: X \to I$ is a fibration and $p: R \to X \times_I X$ is an equivalence relation. Then for any map $h: A \to I$, the set
\[ \{ f: A \to X \, : \, xf = h \} \]
carries an equivalence relation: indeed, we will say that two maps $f, g: A \to X$ from this set are \emph{$R$-equivalent} if there is a map $H: A \to R$ such that $pH = (f, g)$; in this case we will write $f \sim_R g$, or $H: f \sim_R g$ if we wish to make the witness $H$ explicit. It is easily checked that $R$-equivalence defines an equivalence relation on the set  $\{ f: A \to X \, : \, xf = h \}$. We will mainly be interested in the special case where $h: A \to I$ is a fibration as well, in which case this argument shows that each hom-set
\[ {\rm Hom}_{\ct{C}(I)}(A, X) \]
carries an equivalence relation.

\begin{defi}{similareqrel}
Two equivalence relations $p: R \to X \times_I X$ and $q: S \to X \times_I X$ on $x: X \to I$ will be called \emph{similar} if they induce the same equivalence relation on each hom-set ${\rm Hom}_{\ct{C}(I)}(A, X)$.
\end{defi}

\begin{lemm}{eqrelinducingthesameeqrel}
In a tribe \ct{C} two equivalence relations $p: R \to X \times_I X$ and $q: S \to X \times_I X$ are similar if and only if there are maps $H: R \to S$ and $K: S \to R$ such that $qH = p$ and $pK = q$.
\end{lemm}
\begin{proof}
Suppose $p = (p_1,p_2): R \to X \times_I X$ and $q= (q_1, q_2): S \to X \times_I X$ induce the same equivalence relation on each hom-set ${\rm Hom}_{\ct{C}(I)}(A, X)$. Since $p_1$ and $p_2$ are $R$-equivalent arrows in ${\rm Hom}_{\ct{C}(I)}(R, X)$, this means that they must also be $S$-equivalent; hence there is an arrow $H: R \to S$ such that $qH = p$. Similarly, there is an arrow $K: S \to R$ such that $pK = q$. Conversely, postcomposing with $H$ yields a morphism witnessing that two arrows from ${\rm Hom}_{\ct{C}(I)}(A, X)$ are $S$-equivalent provided one starts with a morphism showing that they are $R$-equivalent, while postcomposing with $K$ yields the other direction.
\end{proof}

\section{Propositional identity types}

\subsection{The syntax} We now come to our main syntactic definition. We will say that a type theory has \emph{propositional identity types} if it comes equipped with a type former Id satisfying the rules in Table \ref{table1}.
On top of the rules detailed in this table, we have congruence rules for ${\rm Id}, {\bf r}, {\bf J}, {\bf H}$, which we do not spell out here (see \refrema{butwhataboutcongruence}). If Id is a type former satisfying these rules, we refer to Id as the \emph{propositional identity type} and if we have a term $p \in {\rm Id}_A(s, t)$ we will say that $s$ and $t$ are \emph{propositionally equal} as elements of type $A$. This notion of propositional equality is weaker than the notion of definitional equality we have seen before: indeed, if $s = t \in A$, then ${\bf r}(s) \in {\rm Id}_A(s, t)$ by the introduction and congruence rules.

This means that the rules for the propositional identity types differ from the usual ones in two respects:
\begin{enumerate}
\item The computation rule holds only \emph{propositionally}: that is, it states a propositional instead of a definitional equality.
\item We allow for an additional contextual parameter $\Delta$ in the elimination and computation rules. In the presence of $\Pi$-types, this is equivalent to the rule without this parameter, but,  as observed in \cite[p. 94]{gambinogarner08} and \cite[p. 587]{jacobs99}, in the absence of $\Pi$-types such an additional parameter is essential to establish the basic properties of identity.
\end{enumerate}

\begin{table} \caption{Rules for propositional identity types} \label{table1}
\begin{center}
\begin{tabular}{c}
Formation Rule \\ \\
\AxiomC{$a \in A$}
\AxiomC{$b \in A$}
\BinaryInfC{${\rm Id}_A(a, b) \in \mbox{Type}$}
\DisplayProof \\ \\
Introduction Rule \\ \\
\AxiomC{$a \in A$}
\UnaryInfC{${\bf r}(a) \in {\rm Id}_A(a,a) $}
\DisplayProof \\ \\
Elimination Rule \\ \\
\Axiom$C(x, y, u) \in \mbox{Type} \fCenter\ \, [ \, x \in A, y \in A, u \in {\rm Id}_A(x,y), \Delta(x,y,u) \, ]$ \noLine
\UnaryInf$p \in {\rm Id}_A(a, b) \fCenter\ $ \noLine
\UnaryInf$d(x) \in C(x, x, {\bf r}(x)) \fCenter\ \, [ \,  x \in A, \Delta(x,x,{\bf r}(x)) \, ] $
\UnaryInf${\bf J}(a, b, p,d) \in C(a, b, p) \fCenter\ \, [ \,\Delta(a, b,p) \, ]$
\DisplayProof \\ \\
Computation Rule \\ \\
\Axiom$C(x, y, u) \, \in \mbox{Type} \fCenter \, [ \,x \in A, y \in A, u \in  {\rm Id}_A(x,y), \Delta(x,y,u) \,]$ \noLine
\UnaryInf$a \in A \fCenter $ \noLine
\UnaryInf$d(x) \in C(x, x, {\bf r}(x)) \fCenter \,  [ \,  x \in A, \Delta(x,x,{\bf r}(x)) \, ] $
\UnaryInf${\bf H}(a,d) \in {\rm Id}_{C(a, a, r(a))}({\bf J}(a, a, {\bf r}(a)), d(a)) \fCenter \, [ \, \Delta(a, a, r(a)) \, ]$
\DisplayProof
\end{tabular}
\end{center}
\end{table}

\subsection{Tribes with propositional identity types} Suppose \ct{C} is the classifying category of a type theory with propositional identity types and strong sums. We know that \ct{C} is a tribe, but what more can we say because the type theory has propositional identity types? Using the fact that any fibration is isomorphic to a display map, the rules for propositional identity types gives us:
\begin{enumerate}
\item[(1)$'$] For any fibration $\alpha: A \to I$ there is a factorisation of the diagonal \[ \Delta_A: A \to A \times_I A \] as a map $r: A \to P_I A$ followed by a fibration $(s, t): P_I A \to A \times_I A$, where the first map $r: A \to P_I A$ has the following property: if $f: B \to P_I A$ and $g: C \to B$ are fibrations and $d: r^*B \to r^*C $ is a section of $r^* g$, then there is a section $J: B \to C$ of $g$ and a map $H: r^*B \to P_{r^* B}(r^*C)$ such that $sH = r^*J$ and $tH = d$.
\end{enumerate}
But more than this is true.

Recall that in the classifying category associated to a type theory pullbacks of display maps along arbitrary maps exist and can be computed by making appropriate substitutions. However, substitution is an operation on syntax which preserves syntactic equality: in particular, it preserves all the possible structure strictly. This means that in the classifying category all the structure (by which we mean $P_I, r, (s,t), J$ and $H$) will be preserved on the nose by the pullback functors.

It will not be important for us to demand that the maps $r, (s,t), J, H$ are preserved by change of base. Preservation of $P_I$ will be important, though, but for our purposes preservation up to isomorphism is sufficient, as in:
\begin{enumerate}
\item[(2)$'$] For any map $k: J \to I$ there is an isomorphism between $P_J(k^*A)$ and $k^*(P_IA)$ which is compatible with the isomorphism $k^*(A \times_I A) \cong k^*A \times_J k^*A$.
\end{enumerate}
In fact, an even weaker condition suffices. As we will see below, the morphism $P_I(A) \to A \times_I A$ will turn out to be an equivalence relation on $A \to I$, so in view of \reflemm{eqrelinducingthesameeqrel}, the following, weaker, requirement is arguably more natural:
\begin{enumerate}
\item[(2)] For any map $k: J \to I$ we have morphisms between $P_J(k^*A)$ and $k^*(P_IA)$ which commute over $k^*(A \times_I A) \cong k^*A \times_J k^*A$.
\end{enumerate}

In the presence of (2) the condition (1)$'$ is equivalent to the following requirement, which looks more categorical:
\begin{enumerate}
\item[(1)] For any fibration  $\alpha: A \to I$ there is a factorisation of the diagonal \[ \Delta_A: A \to A \times_I A \] as a map $r: A \to P_I A$ followed by a fibration $(s, t): P_I A \to A \times_I A$. The map $r: A \to P_I A$ has the property that if $g$ is any pullback of $r$ along a fibration and
\diag{ V \ar[d]_g \ar[r]^k & B \ar[d]^f \\
U \ar[r]_l & A }
is a commuting square with a fibration $f$ on the right, then there are maps $d: U \to B$ and $H: V \to P_A(B)$ such that $fd = l, sH = dg, tH = k$ hold.
\end{enumerate}

\begin{lemm}{1equivto1primeinpresenceof2}
Let \ct{C} be a tribe satisfying condition {\rm (2)}. Then the conditions {\rm (1)} and {\rm (1)$'$} are equivalent.
\end{lemm}
\begin{proof}
We show that the following two conditions are equivalent for a map $g: V \to U$:
\begin{enumerate}
\item[(a)] If $h: W \to U$ is a fibration and $\sigma$ is a section of $g^*h: V \times_U W \to V$, then there is a section $J: U \to W$ of $h$ and a map $H: V \to P_V(V \times_U W)$ such that $sH = g^*J$ and $tH = \sigma$.
\item[(b)] If
\diag{ V \ar[d]_g \ar[r]^k & B \ar[d]^f \\
U \ar[r]_l & A }
is a commuting square with a fibration $f$ on the right, then there are maps $d: U \to B$ and $K: V \to P_A(B)$ such that $fd = l, sK = dg$ and $tK = k$ hold.
\end{enumerate}

(a) $\Rightarrow$ (b): Assume we are given a commuting square with $g$ on the left and a fibration $f$  on the right, as in (b), and consider the following double pullback diagram:
\diag{ V \times_A B \ar[r] \ar[d]_{g^*h} \ar@/^1pc/[rr]^{\pi_2} & U \times_A B \ar[d]_h \ar[r]_{\pi_2} & B \ar[d]^f \\
V \ar[r]_g & U \ar[r]_l & A. }
The map $h$, as a pullback of $f$, is a fibration and the map $k$ determines a section $\sigma$ of $g^*h$ such that $\pi_2 \sigma = k$. So, by assumption, $h$ has a section $J$ and there is a map $H: V \to P_V(V \times_A B)$ such that $sH = g^* J$ and $tH = \sigma$. From the fact that $P_V(V \times_A B)$ and $V \times_A P_A(B)$ are similar, it follows that there is a map \[ L: P_V(V \times_A B) \to P_A(B) \] such that $(s, t)L = (\pi_2 s, \pi_2 t)$. So if we put $d := \pi_2 J$ and $K := LH$, then $fd = f\pi_2J = lhJ = l$, $sK = sLH = \pi_2 sH = \pi_2 g^* J  = \pi_2 J g = dg$ and $tK = tLH = \pi_2tH = \pi_2 \sigma = k.$

(b) $\Rightarrow$ (a): Suppose $h: W \to U$ is a fibration and $\sigma$ is a section of $g^*h: V \times_U W \to V$. Then $\sigma = (1, k)$ for some map $k: V \to W$ with $hk = g$. This means that we have a commuting square
\diag{ V \ar[d]_g \ar[r]^k & W \ar[d]^h \\
U \ar[r]_{1_U} & U, }
and hence, by assumption, there is a section $J: U \to W$ of $h$ and a map $K: V \to P_U(W)$ such that $sK = gJ$ and $tK = k$. Because $V \times_U P_U(W)$ and $P_V(V \times_U W)$ are similar,  there is a map
\[ L: V \times_U P_U(W) \to P_V(V \times_U W) \]
such that $(s, t)L = ((\pi_1,s\pi_2), (\pi_1,t\pi_2))$. So if we put $H = L(1,K)$, then
\[ sH = sL(1, K) = (\pi_1,s\pi_2)(1, K) = (1, sK) = (1, gJ) = g^*J \]
and
\[ tH = tL(1, K) = (\pi_1,t\pi_2)(1, K) = (1, tK) = (1, k) = \sigma, \]
as desired.
\end{proof}

Hence we make the following definition.
\begin{defi}{comprcatwithpropidtypes}
We will say that a tribe \ct{C} has \emph{propositional identity types} if it satisfies conditions (1) and (2) above.
\end{defi}

A tribe with propositional identity types is a ``non-split'' model of the propositional identity types, as discussed in the introduction. For future reference we record:

\begin{prop}{classifyingcatwithpropidtypes}
The classifying category of any type theory with strong sums and propositional identity types is a tribe with propositional identity types.
\end{prop}

\subsection{Path structure} In the remainder of this section we will study tribes with propositional identity types. In particular, we will show that they are what we will call path tribes. But before we can define that notion, we first need the definition of path structure.

\begin{defi}{weakpathstr}
Let \ct{C} be a tribe. To equip \ct{C} with \emph{path structure} means specifying for each fibration $x: X \to I$ an equivalence relation \[ (s, t): P_I(X) \to X \times_I X, \]
in such a way that for any map $f: J \to I$ the equivalence relations $P_{J}(f^*X) \to f^*X \times_{J} f^*X$ and $f^*(P_I(X)) \to f^*(X \times_I X) \cong f^*X \times_{J} f^*X$ are similar.
\end{defi}

\begin{lemm}{propidtypeshenceweakpathstructure}
Let \ct{C} be a tribe with propositional identity types. Then \ct{C} has path structure.
\end{lemm}
\begin{proof}
This follows from \reflemm{homiseqrel} in the appendix. Indeed, if \ct{C} is a tribe with propositional identity types and we declare all fibrations to be display maps and all pullback of maps $r: X \to P_IX$ along fibrations to be weak equivalences, then all the axioms (1-5) from the appendix are satisfied. Thus, \reflemm{homiseqrel} applies and we can deduce that $(s, t): P_IX \to X \times_I X$ is always an equivalence relation, as anticipated.
\end{proof}

If \ct{C} is a tribe with path structure, each object $A$ comes equipped with an equivalence relation $(s, t): PA = P_1A \to A \times A$. This implies that each hom-set ${\rm Hom}(B, A)$ carries an equivalence relation; indeed, we will call two maps $f, g: B \to A$ \emph{homotopic} if there is a map $H: B \to PA$ (a \emph{homotopy}) such that $(f, g) = (s, t)H$. In this case we will write $f \simeq g$, or $H: f \simeq g$ if we wish to stress the homotopy $H$. From this definition and the stability property (2) for path structure, the following lemma follows immediately.

\begin{lemm}{homrelisstableeq} Let \ct{C} be a tribe with path structure.
\begin{enumerate}
\item The homotopy relation is stable under precomposition with any map.
\item For each object $X$ the tribe $\ct{C}(X)$ also has path structure, and for every morphism $f: Y \to X$ the change of base functor \[ f^*: \ct{C}(X) \to \ct{C}(Y) \] preserves the homotopy relation.
\end{enumerate}
\end{lemm}

\begin{rema}{notcongruence} Note that we do not claim (yet) that the homotopy relation is a congruence; in particular, we do not claim that the homotopy relation is preserved by postcomposition. This is true in tribes with propositional identity types (we will prove this in \reflemm{standardcomprcatisfibredtribe} below), but it does not seem to be hold in general tribes with path structure.
\end{rema}

To state the definition of a path tribe we need the definition of a contractible map.

\begin{defi}{contractible}
Suppose that $\ct{C}$ is tribe equipped with path structure. A fibration $x: X \to I$ will be called \emph{contractible} if both $x$ and \[ (s, t): P_I(X) \to X \times_I X \] have sections. An object $A$ will be called contractible if $A \to 1$ is contractible.
\end{defi}

Again, the following is immediate from the definition and the stability property (2) for path structure.
\begin{lemm}{contractiblemaps} In a tribe with path structure, contractible fibrations are stable under pullback along arbitrary maps.
\end{lemm}

The following lemma gives an alternative characterisation of contractible maps, which will often prove useful.

\begin{lemm}{altcharcontractibility}
Let \ct{C} be a tribe with path structure. A fibration $x: X \to I$ is contractible if and only if there are maps $f: I \to X$ and $H: X \to P_I(X)$ such that $xf = 1$ and $(s, t)H = (1,fx)$.
\end{lemm}
\begin{proof}
If $x$ is contractible, then it has a section $f$ and $(s, t): P_I(X) \to X \times_I X$ has some section $L$. Writing $H := L(1,fx)$, we get $(s, t)H = (s,t)L(1, fx) = (1, fx)$.

Conversely, suppose $x: X \to I$ has a section $f$ and there is a map $H: X \to P_I(X)$ such that $(s, t)H = (1,fx)$. Since $(s, t): P_I(X) \to X \times_I X$ is an equivalence relation, we also obtain a map $H': X \to P_I(X)$ with $(s,t)H' = (fx, 1)$ by symmetry and a map $\mu: P_I(X) \times_X P_I(X) \to P_I(X)$ with $s\mu = s\pi_1$ and $t\mu = t\pi_2$ by transitivity. Define $L: X \times_I X \to P_I(X)$ by
\[ L := \mu(Hp_1, H'p_2). \]
This map is well-defined since
\[ tHp_1 = fxp_1 = fxp_2 = sH'p_2. \]
In addition, we have
\[ (s, t)L = (s,t)\mu(Hp_1,H'p_2) = (sHp_1,tH'p_2) = (p_1,p_2) = 1, \]
showing that $L$ is a section of $(s, t): P_I(X) \to X \times_I X$.
\end{proof}

\begin{defi}{pathstructure}
A tribe \ct{C} will be called a \emph{path tribe} if it carries path structure in such a way that:
\begin{enumerate}
\item[(i)] all fibrations $s: PX \to X$ are contractible, and
\item[(ii)] if $p: Y \to X$ is a fibration and $Y \times_X PX$ is the pullback
\diag{ Y \times_X PX \ar[d]_{p_1} \ar[r]^(.6){p_2} & PX \ar[d]^s \\
Y \ar[r]_p & X, }
then there is a map $\Gamma: Y \times_X PX \to Y$ such that $p \Gamma = t p_2$.
\end{enumerate}
\end{defi}

\begin{prop}{fromidtopathstructure}
Any tribe with propositional identity types is a path tribe. In particular, the syntactic category associated to a type theory with propositional identity types and strong sums is a path tribe.
\end{prop}
\begin{proof}
Requirement (ii) for a path tribe follows again from the appendix: indeed, it is \reflemm{existencetransport} therein. So it remains to verify that property (i) holds, for which we use \reflemm{altcharcontractibility}.

Since $rs = 1$ it remains to construct a map $H: PX \to P_X(PX)$ such that $(s, t)H = (1,rs)$, where we regard $PX$ as an object in $\ct{C}(X)$ via the source map $s: PX \to X$. The diagram
\diag{ X \ar[r]^(.4){rr} \ar[d]_r & P_X(PX) \ar[d]^{(s, t)} \\
PX \ar[r]_(.35){(1,rs)} & PX \times_X PX }
commutes, so requirement (1) for propositional identity types yields a map $H: PX \to P_X(PX)$ with the desired property.
\end{proof}

\section{Path categories}

\subsection{Definition} We now come to the other main concept of this paper, that of a \emph{path category}. The aim of this section will be to introduce this notion and show, using results from \cite{bergmoerdijk16}, that path categories are tribes with propositional identity types.

A path category consists of a category \ct{C} together with two classes of maps called the \emph{weak equivalences} and the \emph{fibrations}, respectively. Morphisms which belong to both classes of maps will be called \emph{acylic fibrations}. A \emph{path object} on an object $B$ is a factorisation of the diagonal $\Delta: B \to B \times B$ as a weak equivalence $r: B \to PB$ followed by a fibration $(s, t): PB \to B \times B$.

\begin{defi}{fibrationcat} The category \ct{C} will be called a \emph{path category} (short for a \emph{category with path objects}) if the following axioms are satisfied:
\begin{enumerate}
\item Isomorphisms are fibrations and fibrations are closed under composition.
\item The pullback of a fibration along any other map exists and is again a fibration.
\item \ct{C} has a terminal object $1$ and every map $X \to 1$ to the terminal object is a fibration.
\item Isomorphisms are weak equivalences.
\item Weak equivalence satisfy 2-out-of-6: if $f: A \to B$, $g: B \to C$, $h: C \to D$ are three composable maps and both $gf$ and $hg$ are weak equivalences, then so are $f,g,h$ and $hgf$.
\item For any object $B$ there is at least one path space $PB$ (not necessarily functorial in $B$).
\item Every acyclic fibration has a section.
\item The pullback of an acylic fibration along any other map exists is again an acyclic fibration.
\end{enumerate}
\end{defi}

In the paper \cite{bergmoerdijk16} we study these path categories in great detail (with many of the basic results deriving from Brown \cite{brown73}). Here we recall the features of path categories from \cite{brown73,bergmoerdijk16} that will be important for our purposes. They will all be familiar to anyone accustomed with any of the current approaches to abstract homotopy category, such as Quillen model structures.

First of all, the category obtained by inverting the weak equivalences can be constructed very concretely by defining a suitable notion of \emph{homotopy}. Indeed, two parallel arrows $f, g: Y \to X$ will be called \emph{homotopic} if there is a path object $PX$ on $X$ with weak equivalence $r: X \to PX$ and fibration $(s, t): PX \to X \times X$ as well as a morphism $h: Y \to PX$ (the \emph{homotopy}) such that $sh = f$ and $th = g$; in this case we will write $f \simeq g$, or $h: f \simeq g$ if we wish to stress the homotopy. It can be shown that this definition is independent of the choice of path object: that is, if $f$ and $g$ are homotopic via a path object $PX$ and homotopy $h: Y \to PX$, and $P'X$ is another path object with weak equivalence $r': X \to P'X$ and fibration $(s', t'): P'X \to X \times X$, then there is also a homotopy $h': Y \to P'X$ with $s'h' = f$ and $t'h' = g$.

In addition, it can be shown that the homotopy relation is a congruence: it defines an equivalence relation on each hom-set, and composition behaves well with respect to this equivalence relation. This means that one can quotient the category \ct{C} by the homotopy relation: the result is called the \emph{homotopy category} of \ct{C} and is denoted Ho(\ct{C}). The weak equivalences are precisely those morphisms in \ct{C} that become invertible in ${\rm Ho}(\ct{C})$: that is, they coincide with the \emph{homotopy equivalences}. For this reason, Ho(\ct{C}) is the universal solution to inverting the weak equivalences.

Factorisations are another important feature of Quillen model categories. The axioms for a path category only demand that any diagonal $X \to X \times X$ can be factored as a weak equivalence followed by a fibration; however, it can be shown that \emph{any} morphism $f: Y \to X$ in a path category can be factored as a weak equivalence $w_f: Y \to P_f$ followed by a fibration $p_f: P_f \to X$. In fact, one can choose $w_f$ in such a way that it is a section of an acyclic fibration $a_f: P_f \to Y$.

This means in particular that if $p: X \to I$ is a fibration, then the fibrewise diagonal $X \to X \times_I X$ can be factored as a weak equivalence $r: X \to P_I(X)$ followed by fibration $(s, t): P_I(X) \to X \times_I X$. So if $f, g: Y \to X$ are two parallel morphisms and $pf = pg$, we can ask ourselves the question whether there is a morphism $h: Y \to P_I(X)$ such that $sh = f$ and $th = g$. If this is the case, we call $f$ and $g$ \emph{fibrewise homotopic}; this we denote by $f \simeq_I g$ (with the fibration $p: X \to I$ being understood), or $h: f \simeq_I g$ if we again wish to stress the homotopy $h$. As with the ordinary homotopy relation, this can be shown to be independent of the choice of path object and to define an equivalence relation on those classes of morphism that become equal upon postcomposing with $p$.

This fact can be used to show that the notion of a path category is stable under slicing. Indeed, if \ct{C} is a path category and $I$ is an object in \ct{C} one can define a new path category $\ct{C}(I)$: it is the full subcategory of the slice category $\ct{C}/I$ whose objects are the fibrations, while a morphism in $\ct{C}(I)$ is a fibration or a weak equivalence precisely when it is a fibration or a weak equivalence in \ct{C}.

\begin{prop}{pathcatsandslicing} \cite[p.~428]{brown73}
The category $\ct{C}(I)$ is again a path category and for any morphism $f: J \to I$ the pullback functor
\[ f^*: \ct{C}(I) \to \ct{C}(J) \]
preserves fibrations, weak equivalences, the terminal object and pullbacks of fibrations along arbitrary maps.
\end{prop}

This proposition is used by Brown to derive the following additional property of path categories:

\begin{prop}{weakeqstableunderpbkalongfibr} \cite[p.~428]{brown73}
In a path category the weak equivalences are stable under pullback along fibrations.
\end{prop}

Lifting properties form the other main ingredient of Quillen model categories, besides factorisations; indeed, any Quillen model category comes equipped with two weak factorisation systems. Path categories are less well-behaved; indeed, when it comes to lifing properties in path categories the following result from \cite{bergmoerdijk16} seems to be the best possible.

\begin{theo}{weakliftinginfibrcat} \cite[Theorem 2.38]{bergmoerdijk16}
Suppose
\diag{ D \ar[d]_w \ar[r]^l & B \ar[d]^p \\
C \ar[r]_k & A }
is a commuting square in a path category \ct{C} with a weak equivalence $w$ on the left and a fibration $p$ on the right. Then there is a map $d: C \to B$ such that $pd = k$ and $dw \simeq_A l$ (where $\simeq_A$ refers to the fibrewise homotopy relation via the fibration $p$).
\end{theo}

This gives us enough information to derive that path categories are tribes with propositional identity types:

\begin{prop}{fromfibrcattoweakidentitytypes}
Any path category is a tribe with propositional identity types.
\end{prop}
\begin{proof}
We need to check the two conditions for having propositional identity types. Condition (1) is an immediate consequence of the factorisation of any map as a weak equivalence followed by a fibration, \refprop{weakeqstableunderpbkalongfibr} and \reftheo{weakliftinginfibrcat}.

Suppose that in a path category we have two ways of factoring $f: C \to A$ as a weak equivalence followed by a fibration, say $f = pw = p'w'$. Then
\diag{ C \ar[r]^w \ar[d]_{w'} & B \ar[d]^p \\
B' \ar[r]_{p'} & A }
commutes, so \reftheo{weakliftinginfibrcat} implies that there are maps $g: B \to B'$ and $h: B' \to B$ such that $p'g = p$ and $ph = p'$. This means in particular that any two path objects on an object $X$ determine similar equivalence relations. Moreover, \refprop{pathcatsandslicing} implies that path objects are preserved by change of base; from this it follows that in path categories condition (2) for having propositional identity types is satisfied as well.
\end{proof}

\section{Path categories from type theories with strong sums}

The aim of this section is to prove that the syntactic category associated to a type theory with strong sums and propositional identity types carries a path category structure. Since such a syntactic category is a tribe with propositional identity types, it suffices to prove the converse of \refprop{fromfibrcattoweakidentitytypes}: that is, it suffices to show that in a tribe with propositional identity types one can define a class of weak equivalences in such a way that it becomes a path category. In \refprop{fromidtopathstructure} we have proved that tribes with propositional identity types are path tribes; this means that it would be sufficient to prove that in any path tribe one can identify a class of weak equivalences in such a way that it becomes a path category. Indeed, that is what we will do in this section.

Therefore throughout this section $\ct{C}$ will be a path tribe. We have to identify a suitable class of weak equivalences: for these we take the \emph{homotopy equivalences}, defined as follows.

\begin{defi}{homotopyapp} A map $f: X \to Y$ is a \emph{homotopy equivalence} if there is a map $g: Y \to X$ (a \emph{homotopy inverse}) such that the composites $fg$ and $gf$ are homotopic to the identity on $Y$ and $X$, respectively.
\end{defi}

For the proof that with these homotopy equivalences as the weak equivalences \ct{C} becomes a path category, it will be convenient to introduce the auxiliary notion of a left map.

\begin{defi}{leftmap}
A map $l: D \to C$ in \ct{C} will be called a \emph{left map}, if for any commutative square
\diag{ D \ar[r]^n \ar[d]_l & B \ar[d]^p \\
C \ar[r]_m & A }
with a fibration $p$ on the right, there is a map $d: C \to B$ such that $pd = m$ (we will call such a map a \emph{lower filler}).
\end{defi}

\begin{lemm}{basicfactsaboutleftmaps}
\begin{enumerate}
\item If $f: J \to I$ is a fibration, then $\Sigma_f:\ct{C}(J) \to \ct{C}(I)$ preserves and reflects left maps.
\item If $f: Y \to X$ is a map with a homotopy section, that is, a map $g: X \to Y$ such that $fg \simeq 1$, then $f$ is a left map.
\end{enumerate}
\end{lemm}
\begin{proof} (1): From the fact that $\Sigma_f$ has a right adjoint preserving fibrations, it follows that left maps are preserved by $\Sigma_f$.

To show that $\Sigma_f$ reflects left maps, suppose that $l: Y \to X$ fits in a commutative square
\diag{ Y \ar[r]^n \ar[d]_l & B \ar[d]^p \\
X \ar[r]_m & A }
in $\ct{C}(J)$ with a fibration $p$ on the right, while $\Sigma_f(l)$ is a left map. We need to construct a lower filler. By pulling back $p$ along $m$ if necessary, we may assume that $m = 1$. Using that $\Sigma_f$ preserves fibrations, we see that $\Sigma_f(p)$ has a section. But then $p$ has a section as well.

(2): Suppose $f: Y \to X$ fits in a commutative square
\diag{ Y \ar[r]^n \ar[d]_f & B \ar[d]^p \\
X \ar[r]_m & A }
with a fibration $p$ on the right. As we did in (1), we may assume that $m = 1$. Let $g: X \to Y$ and $h: X \to PX$ be such that $sh = fg$ and $th = 1$, and let $\Gamma: B \times_A PA \to B$ be such that $p\Gamma = tp_2$. Putting $d := \Gamma(ng, h)$, we obtain
\[ pd = p\Gamma(ng, h) = tp_2(ng, h) = th = 1 = m, \]
as desired.
\end{proof}

\begin{lemm}{charcontrtypes}
The following are equivalent for an object $A$:
\begin{enumerate}
\item $A$ is contractible.
\item The unique map $!: A \to 1$ is a homotopy equivalence.
\item There is a left map $a: 1 \to A$.
\end{enumerate}
\end{lemm}
\begin{proof}
(1) $\Rightarrow$ (2): If $A$ is contractible, then by \reflemm{altcharcontractibility} there exist a map $k: 1 \to A$ and a homotopy $H: A \to PA$ such that $H: 1 \simeq k!$. Since $!k = 1$, this shows that $!: A \to 1$ is a homotopy equivalence.

(2) $\Rightarrow$ (3): If $!: A \to 1$ is a homotopy equivalence, then it has a homotopy inverse $a: 1 \to A$. This $a$ is a homotopy equivalence as well, hence a left map by the previous lemma.

(3) $\Rightarrow$ (1): If there is a left map $a: 1 \to A$, then we can find a lower filler $H$ for the square
\diag{ 1 \ar[r]^{ra} \ar[d]_a & PA \ar[d]^{(s, t)} \\
A \ar[r]_(.4){(1,a!)} \ar@{.>}[ur]_H & A \times A. }
Such an $H$ is a homotopy showing $1 \simeq a!$ and hence $A$ is contractible by \reflemm{altcharcontractibility}.
\end{proof}

\begin{lemm}{acyclicfibrationsOK} The following are equivalent for a fibration $p: E \to X$:
\begin{enumerate}
\item $p$ is contractible.
\item $p$ has a section which is a left map.
\item $p$ is a homotopy equivalence.
\end{enumerate}
\end{lemm}
\begin{proof} (1) $\Rightarrow$ (2): If $p$ is contractible, then by the previous lemma it has a section which is a left map in $\ct{C}(X)$. Applying $\Sigma_X$ yields a left map which is a section of $p$.

(2) $\Rightarrow$ (3): If $p$ has a section $c$ which is also a left map, then there is a lower filler $h$ for
\diag{ X \ar[d]_c \ar[r]^{rc} & PE \ar[d]^{(s, t)} \\
E \ar[r]_(.4){(1, cp)} \ar@{.>}[ur]_h & E \times E, }
showing that $cp \simeq 1$ and that $p$ is a homotopy equivalence.

(3) $\Rightarrow$ (1): Here we have to be a bit careful as we do not know (yet) that the homotopy relation is preserved by postcomposition; however, we do know that it is preserved by precomposition (see \reflemm{homrelisstableeq} and \refrema{notcongruence}). So suppose that $p: E \to X$ is a homotopy equivalence with homotopy inverse $f$. Since homotopy equivalences are left maps, the square
\diag{ E \ar[r]^1 \ar[d]_p & E \ar[d]^p \\
X \ar[r]_1 & X }
has a lower filler, meaning that $p$ has a section $c$. From $fp \simeq 1$ it follows that
\[ cp \simeq fpcp = fp \simeq 1. \]
Hence $c$ is a homotopy equivalence and a left map in both \ct{C} and $\ct{C}(X)$ by \reflemm{basicfactsaboutleftmaps}. Therefore $p$ is contractible by the previous lemma.
\end{proof}

\begin{lemm}{rleft}
Any map $r: X \to PX$ witnessing reflexivity is a homotopy equivalence and hence a left map.
\end{lemm}
\begin{proof}
By assumption the map $s: PX \to X$ is contractible. But then it follows from the previous lemma that $s$ is a homotopy equivalence with some homotopy inverse $s^{-1}$. This implies that for any $r: X \to PX$ with $sr = 1$ we must have
\[ rs \simeq s^{-1}srs = s^{-1}s \simeq 1, \]
showing that $r$ is a homotopy equivalence as well.
\end{proof}

\begin{lemm}{standardcomprcatisfibredtribe}
The homotopy relation is a congruence, and hence the homotopy equivalences satisfy 2-out-of-6.
\end{lemm}
\begin{proof}
In view of \reflemm{homrelisstableeq} it suffices to show that the homotopy relation is preserved by postcomposition. To see this, note that for any map $f: X \to Y$ there is a commutative square of the form
\diag{ X \ar[d]_r \ar[r]^f & Y \ar[r]^r & PY \ar[d]^{(s, t)} \\
PX \ar[r]_(.4){(s, t)} & X \times X \ar[r]_{f \times f} & Y \times Y}
with a left map on the left and a fibration on the right. So there is a map $Pf: PX \to PY$ such that $(s, t)Pf = (fs, ft)$. This shows that the homotopy relation is also preserved by postcomposition and hence a congruence. To show that the homotopy equivalences satisfy 2-out-of-6, one simply observes that for any congruence the class of morphisms that become isomorphisms in the quotient satisfies 2-out-of-6.
\end{proof}

We conclude:
\begin{prop}{frompathstrtofibredfibrcat}
Let \ct{C} be a path tribe. With the homotopy equivalences as the weak equivalences \ct{C} also has the structure of a path category.
\end{prop}
\begin{proof}
Axioms (1--3) follow from the fact that \ct{C} is a tribe. Axiom (4) follows from the fact that the homotopy relation is reflexive, while axiom (5) was \reflemm{standardcomprcatisfibredtribe}. Axiom (6) follows from \reflemm{rleft}, while axiom (7) follows from \reflemm{acyclicfibrationsOK} and axiom (8) follows from \reflemm{acyclicfibrationsOK} and \reflemm{contractiblemaps}.
\end{proof}

To summarise, we have shown that:
\begin{theo}{pathtribeisfibcat}
The following are equivalent for a tribe \ct{C}:
\begin{enumerate}
\item \ct{C} has propositional identity types.
\item \ct{C} is a path tribe.
\item One can identify a class of weak equivalences on \ct{C} which give \ct{C} the structure of a path category.
\end{enumerate}
\end{theo}
\begin{proof}
 (1) $\Rightarrow$ (2) was \refprop{fromidtopathstructure}, (2) $\Rightarrow$ (3) was \refprop{frompathstrtofibredfibrcat}, while (3) $\Rightarrow$ (1) was \refprop{fromfibrcattoweakidentitytypes}.
\end{proof}

\begin{coro}{classcatofsttypetheory}
The classifying category of a type theory with strong sums and propositional identity types carries the structure of a path category.
\end{coro}
\begin{proof} This follows from the previous theorem and \refprop{fromidtopathstructure}.
\end{proof}

\section{Path categories from general type theories}

In this section we will generalise \refcoro{classcatofsttypetheory} and show that the classifying category of any type theory with propositional identity types has the structure of a path category.

But before we do this we will first introduce some terminology. Recall that the dependent projections are those context morphisms in the classifying category which project away some types at the end of a context. This means that the dependent projections can be stratified into different levels, depending on how many types get projected away. Indeed, we will call a dependent projection in the syntactic category an \emph{$n$-display map} if it projects away $n$ types; we will also say that the \emph{rank} of the dependent projection is $n$. If $X \to 1$ is an $n$-display map (in other words, $X$ is a context of length $n$), then we will say that the object $X$ has rank $n$. Instead of 1-display map we will often simply say \emph{display map} and instead of object of rank 1, we will often simply say \emph{type}.

What additional structure does the syntactic category have if it comes equipped with propositional identity types? Translating the syntax into categorical terms we obtain the following:
\begin{quote}
($\clubsuit$) If $X \to I$ is a display map, then the diagonal $X \to X \times_I X$ can be factored as a map $r: X \to P_IX$ followed by a display map $(s, t): P_IX \to X \times_I X$. This choice of $P_I(X)$ is stable in the sense that if $f: J \to I$ is any map then $P_J(f^*X)$ and $f^*P_I(X)$ are isomorphic. In addition, the map $r$ has the property that if $g$ is any pullback of it along a fibration and \diag{ V \ar[r]^k \ar[d]_{g} & B \ar[d]^p \\
U \ar[r]_l & A, }
is any commutative square with a display map $p$ on the right, then there are maps $d: U \to B$ and $H: V \to P_A(B)$ such that $p d = l$ and $(s, t)H = (gd, k)$.
\end{quote}
Indeed, what we will do now is assume that we are given a tribe \ct{C} such that:
\begin{itemize}
\item[--] For each natural $n \in \NN$ there is a class of $n$-display maps and the classes of $n$-display maps and $m$-display maps are disjoint if $n \not= m$.
\item[--] All $n$-display maps are fibrations and for every fibration $f$ there is some natural number $n$ and $n$-display map $g$ such $f$ and $g$ are isomorphic.
\item[--] The only 0-display maps are the identities.
\item[--] The class of $n$-display maps is stable under pullback.
\item[--] If $f$ is an $n$-display map and $g$ is an $m$-display map, then $fg$ is an $(n+m)$-display map; conversely, if $h$ is an $(n+m)$-display map, then there exist unique $f$ and $g$ such that $h = fg$ with $f$ being an $n$-display map and $g$ being an $m$-display map.
\item[--] The property $(\clubsuit)$ holds.
\end{itemize}
Note that if \ct{C} has this structure, then so does $\ct{C}(X)$ for any object $X$. Our task will be to show that \ct{C} is a path category. We do this by showing that \ct{C} is a path tribe and appealing to \reftheo{pathtribeisfibcat}. In the process we will call maps of the form $r: X \to P_IX$ for display maps $X \to I$ as well as their pullbacks along fibrations \emph{weak equivalences}. The reason for this is that assumption $(\clubsuit)$ implies that these weak equivalences together with the display maps satisfy the axioms whose consequences we study in the appendix. Indeed, in this section we will often use results from the appendix.

\begin{lemm}{syntacticathas1levelpathstr} The category \ct{C} carries path structure in such a way that if $g$ is a weak equivalence, $f$ is a fibration and
\diag{ V \ar[d]_{g} \ar[r]^k & B \ar[d]^f \\
U \ar[r]_l  & A }
commutes, then there are maps $d: U \to B$ and $H: V \to P_A(B)$ such that $fd = l$ and $H: dg \simeq_A k$.
\end{lemm}
\begin{proof}
The idea of the proof is to show the following statement by induction on $n$:
\begin{quote}
For each $n$-display map $B \to A$ one can define an $n$-display map $P_AB \to B \times_A B$ which is an equivalence relation and is such that for any weak equivalence $g$ and commuting square
\diag{ V \ar[d]_{g} \ar[r]^k & B \ar[d]^f \\
U \ar[r]_l  & A }
there are maps $d: U \to B$ and $L: V \to P_A(B)$ such that $pf = l$ and $L: k \simeq_A dg$.
\end{quote}
The only $0$-display maps are identities, so this statement is trivial for $n = 0$.

Now suppose that the statement above holds for $n$; we will show it holds for $n+1$ as well. So let $Y \to I$ be an $(n+1)$-display map; since all structure is stable under slicing, we may just as well assume that $I = 1$. This means that there is an $n$-display map $p: Y \to X$ to a type $X$. The map $(1, rp): Y \to Y \times_X PX$ is a weak equivalence by \reflemm{recurringpullback} from the appendix, so we can apply the induction hypothesis to the diagram
\diag{  Y  \ar[r]^1 \ar[d]_{(1, rp)} & Y \ar[d]^p \\
Y \times_X PX  \ar[r]_(.6){tp_2} & X, }
yielding a transport structure $\Gamma: Y \times_X PX \to Y$ together with a homotopy $H: Y \to P_X(Y)$ such that $p \Gamma = tp_2$ and $tH = 1$ and $sH = \Gamma(1, rp)$. This means that we are in a position to apply \reftheo{mainconstrforTconnections} to $p$ and $(s, t): P_X(Y) \to Y \times_X Y$ and construct a new equivalence relation $PY \to Y \times Y$ by taking two pullbacks:
\diag{ P_X(Y) \ar[d] & & PY \ar[ll] \ar[d] \\
Y \times_X Y & & Y \times_X PX \times_X Y \ar[r]^(.6){p_2} \ar[d]_{(p_1,p_3)} \ar[ll]^(.55){(\Gamma(p_1,p_2),p_3)} & PX \ar[d]^{(s, t)} \\
& & Y \times Y \ar[r]_{f \times f} & X \times X.}
Writing $(\sigma, \tau): PY \to Y \times Y$ for the map down the middle, one sees that it is an $(n+1)$-display map, as desired. Alternatively, one may construct $PY$ as the pullback
\diag{ PY \ar[r]^{q_3} \ar[d]_{(q_1, q_2)} & P_X(Y) \ar[d]^{s} \\
Y \times_X PX \ar[r]_(.6){\Gamma} &  Y,}
with $(\sigma, \tau) = (q_1, tq_3)$. It is this second presentation that we will use below.

Now suppose $g: V \to U$ is a weak equivalence fitting into a commutative square
\diag{ V \ar[r]^k \ar[d]_{g} & C \ar[d]^f \\
U \ar[r]_l & A, }
with an $(n+1)$-display map $f: C \to A$ on the right. The proof will be finished once we show that one may construct a map $d: U \to C$ such that $fd = l$ together with a homotopy $L: V \to P_A(C)$ such that $L: k \simeq_A dg$.

We factor $f$ as $pq$ where $p: B \to A$ is a display map and $q: C \to B$ is an $n$-display map. Our assumption $(\clubsuit)$ applied to
\diag{ V \ar[d]_{g} \ar[r]^k & C \ar[r]^q & B \ar[d]^p \\
U  \ar[rr]_l & & A }
yields a map $e: U \to B$ and a homotopy $K: V \to P_A(B)$ such that $l = pe$ and $K: qk \simeq_A e g$. By \reflemm{recurringpullback} again, the map $(1, rq): C \to C \times_B P_A(B)$ is a weak equivalence, so the induction hypothesis applied to
\diag{   C  \ar[r]^1 \ar[d]_{(1, rq)} & C \ar[d]^q \\
C \times_B P_A(B) \ar[r]_(.65){tp_2} & B, }
in $\ct{C}(B)$ yields a transport structure \[ \Gamma: C \times_B P_A(B) \to C \] such that $q \Gamma = tp_2$. Let $k' := \Gamma(k,K): V \to C$. Then
\[ qk' = q\Gamma(k, K) = tp_2(k, K) = tK = eg, \]
so
\diag{ V \ar[d]_{g} \ar[r]^{k'} & C \ar[d]^q \\
U \ar[r]_e & B }
commutes. Applying the induction hypothesis to $q$ again, but now in \ct{C}, one obtains a map $d: U \to C$ such that $qd = e$ together with a homotopy $H: k' \simeq_B dg$. Note that we have $fd = pqd = pe = l$, so it remains to show that $dg \simeq_A k$.

By construction $P_A(C)$ is the pullback
\diag{ P_A(C) \ar[d]_{(q_1, q_2)} \ar[r]^{q_3} & P_{B}(C) \ar[d]^s \\
C \times_{B} P_A(B) \ar[r]_(.6)\Gamma & C,}
so we have a map $L: Y \to P_A(C)$ given by $L := (k, K, H)$. Then
\[ \sigma L = q_1(k, K, H) = k \]
and
\[ \tau L = tq_3(k, K, H) = tH = dg. \]
This completes the induction step.

It now follows that \ct{C} has path structure: because all $P_A(B) \to B \times_A B$ are equivalence relations, requirement (1) for path structure is satisfied. Requirement (2) follows the stability condition in ($\clubsuit$) and \refprop{dingetje} in the appendix.
\end{proof}

\begin{lemm}{liftsinsyntacticcat}
For any fibration $f: Y \to X$ there is a transport map $\Gamma: Y \times_X PX \to Y$ with $f \Gamma = tp_2$.
\end{lemm}
\begin{proof}
Without loss of generality we may assume that $X \to 1$ is an $n$-display map for some $n$. So we can prove the lemma by induction on the rank $n$ of $X$, with $n = 0$ being trivial.

So suppose $X$ has rank $n+1$. Then there is an $n$-display map $p: X \to A$ whose codomain $A$ is a type. It follows from the  previous proof that there is a transport map \[ M: X \times_A PA \to X \] with $p M = tp_2$ and $1 \simeq_A M(1, rp)$, which is used in the construction of $PX$ as the pullback:
\diag{ PX \ar[d] \ar[r] & P_A(X) \ar[d]^s \\
X \times_A PA \ar[r]_(.6)M & X. }
In addition, the induction hypothesis applied to $f$ in $\ct{C}(A)$ yields a map
\[ N: Y \times_X P_A(X) \to Y \]
with $f N = tp_2$.

From $1 \simeq_A M(1, rp)$ it follows that there is a homotopy \[ H: f \simeq_A M(1, rp)f = M (f \times_A 1) (1,rpf). \] Writing $h := N(1, H)$ this means that there is a commutative square of the form
\diag{ Y \ar[rr]^h \ar[d]_{(1, rpf)} & & Y \ar[d]^f \\
Y \times_A PA \ar[r]_(.48){f \times_A 1} & X \times_A PA \ar[r]_(.6)M & X. }
Since $(1, rpf): Y \to Y \times_A PA$ is a weak equivalence by \reflemm{recurringpullback}, the previous lemma yields a map $l: Y \times_A PA \to Y$ such that $fl = M(f \times_A 1)$.

We have to construct a map $\Gamma: Y \times_X PX \to Y$ with $f \Gamma = tp_2$, where  $Y \times_X PX$ is isomorphic to the pullback
\diag{ Y \times_A PA \times_X P_A(X) \ar[d]_{(q_1, q_2)} \ar[rr]^{q_3} & & P_A(X) \ar[d]^s \\
Y \times_A PA \ar[r]_(.45){f \times_A 1} & X \times_A PA \ar[r]_(.6)M & X. }
We put $\Gamma := N(l(q_1,q_2), q_3)$. This is well-defined, as \[ fl(q_1,q_2) = M(f \times_A 1)(q_1,q_2) = sq_3. \] Morover,
\[ f \Gamma = f N (l(q_1, q_2), q_3) =tp_2(l(q_1, q_2), q_3) = tq_3 = tp_2, \]
as desired.
\end{proof}

\begin{lemm}{acyclicfibrationsclosedundercompinsyntacticcat}
Contractible fibrations are closed under composition.
\end{lemm}
\begin{proof}
Suppose $q: Y \to X$ and $p: X  \to I$ are contractible fibrations. We want to show that $pq$ is contractible as well; for that it suffices to consider the case where $p$ is a display map and $I = 1$. But in that case the result follows from the construction of path objects in \reflemm{syntacticathas1levelpathstr} above and \refprop{morehappy} from the appendix.
\end{proof}

\begin{lemm}{sourcemapsareacyclicinsyntcat}
Every source map $s: PY \to Y$ is contractible.
\end{lemm}
\begin{proof}
We prove this by induction on the rank $n$ of $Y$, with the case $n = 0$ being trivial.

If $Y$ has rank $n + 1$, then there exists an $n$-display map $p: Y \to X$ to a type $X$. From the construction of $PY$ in \reflemm{syntacticathas1levelpathstr} we get that the source map on $Y$ is the arrow $p_1(q_1,q_2)$ down the middle in
\diag{ & PY \ar[d]_{(q_1, q_2)} \ar[r]^{q_3} & P_{X}(Y) \ar[d]^s \\
PX \ar[d]_s & Y \times_{X} PX \ar[r]_\nabla \ar[l]_(.6){p_2} \ar[d]^{p_1} & Y  \\
X & Y \ar[l]^p }
Since both squares in this diagram are pullbacks and contractible fibrations are stable under pullback by \reflemm{contractiblemaps} and closed under composition by the previous lemma, this arrow down the middle is contractible as soon as $s: PX \to X$ is contractible (the map $s: P_X(Y) \to Y$ being contractible by induction hypothesis). However, for $r: X \to PX$ we have $sr = 1$ and from \reflemm{syntacticathas1levelpathstr} it follows that the diagram
\diag{  X \ar[d]_r \ar[r]^(.4){rr} &   P_X(PX) \ar[d]^{(s, t)} \\
PX \ar[r]_(.35){(1, rs)}  & PX \times_X PX }
has a lower filler. Therefore $s: PX \to X$ is contractible by \reflemm{altcharcontractibility}.
\end{proof}

We conclude:
\begin{theo}{classcatofgentypetheory}
The classifying category of any type theory with propositional identity types carries the structure of a path category.
\end{theo}

\appendix

\section{Technical results} In this appendix we collect some technical results that were needed at various points in the main text; often the point is that we are able to prove standard results from homotopy theory in a very weak context, weaker even than that of a path category. In order to do this somewhat systematically, we have decided to derive them in a uniform setting.

This setting is that we are given a tribe \ct{C}. In addition, we are given two classes of maps, called \emph{display maps} and \emph{weak equivalences}, respectively. If $A \to 1$ is a display map, we call $A$ a \emph{type}. We will make the following assumptions:
\begin{enumerate}
\item Display maps are fibrations (but the converse need not hold).
\item For any map $m: C \to A$ and display map $f: B \to A$ there is a pullback square
\diag{D \ar[r]^n \ar[d]_g & B \ar[d]^f \\
C \ar[r]_m & A }
in which $g$ is a display map as well.
\item Any pullback of a weak equivalence along a fibration is again a weak equivalence.
\item If $f: B \to A$ is a display map, then the fibrewise diagonal $B \to B \times_A B$ factors as a weak equivalence $r: B \to P_A(B)$ followed by a display map $(s, t): P_A(B) \to B \times_A B$. (We will refer to $P_A(B)$ together with $r,s,t$ as a \emph{path object} for $f$.)
\item If
\diag{ D \ar[d]_w \ar[r]^n & B \ar[d]^p \\
C \ar[r]_m & A }
is a commutative square in which $w$ is a weak equivalence and $p$ is a display map, then there are maps $d: C \to B$ and $H: D \to P_A(B)$ such that $pd = m, sH = dw, tH = n$.
\end{enumerate}
Note that if \ct{C} is a tribe with this structure, then so is any $\ct{C}(X)$.

We do not believe that this setting is so interesting in itself, but, as said, by organising matters in this way we are able to derive the results we need in a uniform and systematic way. 

\subsection{Groupoid structure} Here we show that types carry a groupoid structure ``up to homotopy''.

\begin{lemm}{recurringpullback} If $f: Y \to X$ is a fibration whose codomain is a type and $Y \times_X PX$ is the pullback
\diag{ Y \times_X PX \ar[d]_{p_1} \ar[r]^(.6){p_2} & PX \ar[d]^s \\
Y \ar[r]_f & X ,}
then \[ (1, rf): Y \to Y \times_X PX \] is a weak equivalence.
\end{lemm}
\begin{proof}
This is because
\diag{ Y \ar[d]_{(1, rf)} \ar[r]^f & X \ar[d]^r \\
Y \times_X PX \ar[r]_(.6){p_2} \ar[d]_{p_1} & PX \ar[d]^s \\
Y \ar[r]_f & X }
consists of pullbacks in which $f$ and $p_2$ are fibrations and $r$ is a weak equivalence.
\end{proof}

\begin{lemm}{homiseqrel}
If $A$ is a type, then $(s, t): PA \to A \times A$ is an equivalence relation.
\end{lemm}
\begin{proof}
We have $r: A \to PA$ for reflexivity. To witness symmetry, note that
\diag{ A \ar[r]^r \ar[d]_r & PA \ar[d]^{(s, t)} \\
PA \ar[r]_(.45){(t, s)} & A \times A }
is a commuting square with a weak equivalence on the left and a display map on the right. Therefore we have a map $\sigma: PA \to PA$ such that $(s, t)\sigma = (t, s)$. In addition, the previous lemma together with the commutativity of
\diag{ PA \ar[rr]^1 \ar[d]_{(1, rt)} & & PA \ar[d]^{(s, t)} \\
PA \times_A PA \ar[rr]_{(sp_1, tp_2)} & & A \times A, }
gives us a map $\mu$ such that $(s\pi_1,t\pi_2) = (s, t)\mu$.
\end{proof}

It follows from the previous lemma that if $B$ is arbitrary and $A$ is a type, then the homset ${\rm Hom}(B, A)$ carries an equivalence relation: indeed, two such parallel maps $f, g: B \to A$ will be equivalent is there is a map $H: B \to PA$ such that such $sH = f$ and $tH = g$. In this case we call $f$ and $g$ \emph{homotopic} and $H$ a \emph{homotopy} and we write $f \simeq g$, or $H: f \simeq g$ if we wish to stress the homotopy.

More generally, if $p: A \to X$ is a display map and $k: B \to X$ is arbitrary, then the set
\[ \{ f: B \to A \, : \, pf = k \} \]
carries an equivalence relation as well. Indeed, two such maps $f, g: B \to A$ will be equivalent in case there is a map $H: B \to P_X(A)$ such that $sH = f$ and $tH = g$.  In this case $f$ and $g$ are \emph{fibrewise homotopic} and $H$ is a \emph{fibrewise homotopy} and we will write $f \simeq_X g$ or $H: f \simeq_X g$.

Clearly, the homotopy relation is preserved by precomposition. We also have that it is preserved by postcomposition in the following sense:

\begin{lemm}{homandpostcomp}
Suppose $f, g: C \to B$ are parallel maps and $h: B \to A$ is a map between types. Then $f \simeq g$ implies $hg \simeq hf$.
\end{lemm}
\begin{proof}
The square
\diag{ B \ar[d]_r \ar[r]^h & A \ar[r]^r & PA \ar[d]^{(s, t)} \\
PB \ar[r]_(.45){(s, t)} & B \times B \ar[r]_{h \times h} & A \times A }
commutes, so we obtain a map $K: PB \to PA$ such that $(s, t)K = (h \times h)(s, t)$. So if $H: C \to PB$ is such that $(s, t)H = (f, g)$, then
\[ (s, t)KH = (h \times h)(s, t)H = (hf, hg). \]
\end{proof}

In addition, we have the following two lemmas:

\begin{lemm}{sueful2}
Suppose $x: X \to I$ and $z: Z \to I$ are display maps, and $f, g: Y \to X$ and $h: Y \to Z$ are maps such that $zh = xf = xg$. If $f \simeq_I g$, then $(f, h) \simeq_I (g, h)$.
\end{lemm}
\begin{proof}
In the diagram
\diag{ X \ar[d]_r & X \times_I Z \ar[l] \ar[d]_{r \times_I 1} \ar[rr]^r & & P_I(X \times_I Z) \ar[d]^{(s, t)} \\
P_IX & P_I X \times_I Z \ar[rr]_(.4){(s \times_I 1, t \times_I 1)} \ar[l] & & (X \times_I Z) \times_I (X \times_I Z) }
the left square is a pullback with a fibration at the bottom, so there is a map $K: P_IX \times_I Z \to P_I(X \times_I Z)$ such that $(s, t)K = (s \times_I 1, t \times_I 1)$. So if $H: Y \to P_IX$ is such that $(s, t)H = (f, g)$, then \[ (s, t)K(H, h) = (s \times_I 1, t \times_I 1)(H, h) = ((sH, h), (tH, h)) = ((f, h), (g, h)). \]
Hence $K(H, h): Y \to P_I(X \times_I Z)$ is a homotopy showing $(f, h) \simeq_I (g, h)$.
\end{proof}

\begin{lemm}{sueful}
If $fh \simeq gh$ and $h$ is a weak equivalence, then $f \simeq g$.
\end{lemm}
\begin{proof}
Suppose $f, g: Y \to X$ are two parallel maps, $h: Z \to Y$ is a weak equivalence and $H: Z \to PX$ is a homotopy with $(s, t)H = (fh, gh)$. Then
\diag{ Z \ar[d]_h \ar[r]^H & PX \ar[d]^{(s, t)} \\
Y \ar[r]_(.4){(f, g)} & X \times X }
commutes and a  lower filler in this diagram is a homotopy showing that $f \simeq g$.
\end{proof}

Clearly, a similar statement as in the previous lemma holds for the notion of fibrewise homotopy.

With these results in place, let us return to \reflemm{homiseqrel}. The proof of this lemma actually yields more than that $(s, t): PA \to A \times A$ is an equivalence relation. Indeed, it also tells us that $\sigma r \simeq_{A \times A} r$ and $\mu(1, rt) \simeq_{A \times A} 1$. Note that from the later one can derive that
\[ \mu(r, r) \simeq_{A \times A} \mu(1, rt)r \simeq_{A \times A} r. \]
This already goes some way towards proving:
\begin{prop}{fibrantgroupoid}
Any type carries a groupoid structure up to homotopy. More precisely, if $A$ is a type then we have that:
\begin{enumerate}
\item $\mu(p_1,\mu(p_2,p_3)) \simeq_{A \times A} \mu(\mu(p_1,p_2),p_3): PA \times_A PA \times_A PA \to PA$.
\item $\mu(1, rt) \simeq_{A \times A} 1: PA \to PA$.
\item $\mu(rs,1) \simeq_{A \times A} 1: PA \to PA$.
\item $\mu(1,\sigma) \simeq_{A \times A} rs: PA \to PA$.
\item $\mu(\sigma, 1) \simeq_{A \times A} rt: PA \to PA$.
\end{enumerate}
\end{prop}
\begin{proof} We take each of these points in turn, making heavy use of the three lemmas we just proved.
\begin{enumerate}
\item \reflemm{recurringpullback} gives us that $w = (1, rtp_2): PA \times_A PA \to PA \times_A PA \times_A PA$ is a weak equivalence. Because \[ \mu(\mu \times_X 1)w = \mu(\mu, rtp_2)= \mu(\mu, rt\mu) = \mu(1,rt)\mu \simeq_{A \times A} \mu \] and \[ \mu(1 \times_X \mu)(1, rtp_2) = \mu(p_1, \mu(p_2, rtp_2)) = \mu(p_1, \mu(1, rt)p_2) \simeq_{A \times A} \mu(p_1, p_2) = \mu, \] the associativity of $\mu$ follows from the previous lemma.
\item was already proved in \reflemm{homiseqrel}.
\item From $\mu(rs, 1)r = \mu(r,r) \simeq_{A \times A} r$ and the previous lemma we deduce that $\mu(rs, 1) \simeq_{A \times A} 1$, as desired.
\item From $\mu(1,\sigma)r = \mu(r,\sigma r) \simeq_{A \times A} \mu(r, r) \simeq_{A \times A} r = (rs)r$ and the previous lemma we deduce $\mu(1,\sigma) \simeq_{A \times A} rs$.
\item is similar to (4).
\end{enumerate}
\end{proof}

\subsection{Constructing equivalence relations.}
A crucial fact is that fibrations allow for some notion of transport.
\begin{lemm}{existencetransport}
Let $f: Y \to X$ be a display map whose codomain $X$ is a type. Then there is a map $\Gamma: Y \times_X PX \to Y$ such that $f\Gamma = tp_2$ and $\Gamma(1, rf) \simeq_X 1$.
\end{lemm}
\begin{proof}
The square
\diag{ Y \ar[d]_{(1, rf)} \ar[r]^1 & Y \ar[d]^f \\
Y \times_X PX \ar[r]_(.6){tp_2} & X }
commutes, so this follows from \reflemm{recurringpullback}.
\end{proof}

In the remainder of this subsection, we will study a more general situation. In fact, we will assume that:
\begin{enumerate}
\item[(a)] We are given a fibration $f: Y \to X$ whose codomain $X$ is a type.
\item[(b)] There is an equivalence relation $\tau: T \to Y \times_X Y$.
\item[(c)] There are maps  $\Gamma: Y \times_X PX \to Y$ and $H: Y \to T$ such that  $f\Gamma = t p_2$ and $\tau H = (1, \Gamma(1, rf))$ (we will call such a map $\Gamma$ a \emph{$T$-transport}).
\item[(d)] Any square with $\tau$ on the right and a weak equivalence on the left has a lower filler.
\end{enumerate}

We will write $\tau_1 := p_1\tau$ and $\tau_2 := p_2\tau$ for the two maps $T \to Y$.

\begin{lemm}{toshowTequiv}
Suppose that $m, n: Z \to Y$ are such that $fm = fn$. If $w: Z' \to Z$ is a weak equivalence and $mw \sim_T nw$, then $m \sim_T n$.
\end{lemm}
\begin{proof}
If $mw \sim_T nw$, there is a map $K: Z' \to T$ such that
\diag{ Z' \ar[r]^K \ar[d]_w & T \ar[d]^\tau \\
Z \ar[r]_(.35){(m, n)} & Y \times_X Y }
commutes. Assumption (d) tells us that this diagram has a lower filler and hence we can deduce that $m \sim_T n$.
\end{proof}

\begin{lemm}{Tconnunique}
$T$-transports are unique up to $T$-equivalence; more precisely, if $\Gamma$ and $\Gamma'$ are two $T$-transports, then $\Gamma \sim_T \Gamma'$.
\end{lemm}
\begin{proof}
If $\Gamma$ and $\Gamma'$ are both $T$-transports, then $\Gamma(1, rf)$ and $\Gamma'(1, rf)$ will be $T$-equivalent, as they are both $T$-equivalent to the identity on $Y$. Since $(1, rf)$ is a weak equivalence, the desired statement now follows from the previous lemma.
\end{proof}

\begin{lemm}{TconnpresTequivalence}
$T$-transports preserve $T$-equivalence; more precisely, if $\Gamma$ is a $T$-transport, the two maps \[ \Gamma(\tau_1p_1, p_2), \Gamma(\tau_2p_1,p_2): T \times_X PX \to Y \] are $T$-equivalent.
\end{lemm}
\begin{proof}
The map $(1,rf\tau_1)$ is a weak equivalence by \reflemm{recurringpullback}, so it suffices to prove that $\Gamma(\tau_1p_1, p_2)$ and $\Gamma(\tau_2p_1,p_2)$ become $T$-equivalent after precomposing with this map. However, we have
\[ \Gamma(\tau_1p_1, p_2)(1,rf\tau_1) = \Gamma(\tau_1,rf\tau_1) = \Gamma(1, rf)\tau_1 \sim_T \tau_1 \]
and
\[ \Gamma(\tau_2p_1, p_2)(1,rf\tau_1) = \Gamma(\tau_2,rf\tau_1) = \Gamma(\tau_2,rf\tau_2) =\Gamma(1, rf)\tau_2 \sim_T \tau_2, \]
while $\tau_1 \sim_T \tau_2$ is true (almost) by definition.
\end{proof}

\begin{lemm}{Tconnstableunderhom}
If $\Gamma$ is a $T$-transport, the two maps \[ \Gamma(p_1,sp_2), \Gamma(p_1, tp_2): Y \times_X P_{X \times X}(PX) \to Y \] are $T$-equivalent.
\end{lemm}
This lemma should be understood as saying the following: if $\Gamma$ is a $T$-transport  and $\alpha$ and $\beta$ are two paths with endpoints $x_0$ to $x_1$ and $\alpha$ and $\beta$ are homotopic relative those endpoints, then for any $y \in Y$ with $f(y) = x_0$ the elements $\Gamma(y, \alpha)$ and $\Gamma(y, \beta)$ will be $T$-equivalent.
\begin{proof}
The map
\[ 1 \times_X r: Y \times_X PX \to Y \times_X P_{X \times X}(PX) \]
is the pullback along the projection and fibration $Y \times_X P_{X \times X}(PX) \to P_{X \times X}(PX)$ of the weak equivalence $r: PX \to P_{X \times X}(PX)$, and hence a weak equivalence as well. Therefore to show that $\Gamma(p_1,sp_2)$ and $\Gamma(p_1, tp_2)$ are $T$-equivalent it suffices to show that they become $T$-equivalent after precomposing with $1 \times_X r$. However,
\[ \Gamma(p_1,sp_2)(1 \times_X r) = \Gamma(p_1, p_2) = \Gamma(p_1, tp_2)(1 \times_X r), \]
so after precomposing with $1 \times_X r$ these maps actually become equal and the lemma follows.
\end{proof}

\begin{lemm}{Tconnandcomposition}
If $\Gamma$ is a $T$-transport and $\mu: PX \times_X PX \to PX$ is a composition on $PX$, then the two maps \[ \Gamma(1 \times_X \mu), \Gamma(\Gamma \times_X 1): Y \times_X PX \times_X PX \to Y \] are $T$-equivalent.
\end{lemm}
This lemma says: if $\Gamma$ is a $T$-transport, $\alpha$ and $\beta$ are two composable paths, and $y \in Y$ is such that $f(y) = s(\alpha)$, then $\Gamma(y, \mu(\alpha, \beta))$ and $\Gamma(\Gamma(y, \alpha), \beta)$ are $T$-equivalent.
\begin{proof}
Recall that $\mu$ being a composition on $PX$ means that $\mu(1, rt) \simeq_{X \times_I X} 1$.

By \reflemm{recurringpullback} the map
\[ (p_1,p_2,rtp_2) = (1, rtp_2): Y \times_X PX \to Y \times_X PX \times_X PX \]
is a weak equivalence, so it suffices to show that $\Gamma(1 \times_X \mu)(p_1,p_2,rtp_2)$ and $\Gamma(\Gamma \times_X 1)(p_1,p_2,rtp_2)$ are $T$-equivalent. However,
\[  \Gamma(\Gamma \times_X 1)(p_1,p_2,rtp_2) = \Gamma(\Gamma(p_1, p_2),rtp_2) = \Gamma(1, rf)\Gamma(p_1,p_2) \sim_T \Gamma(p_1,p_2)= \Gamma, \]
and
\[ \Gamma(1 \times_X \mu)(p_1,p_2,rtp_2) = \Gamma(p_1, \mu(p_2, rtp_2)) \sim_T \Gamma(p_1,p_2) = \Gamma \]
by \reflemm{Tconnstableunderhom}, so the lemma follows.
\end{proof}

We now come to the main point of this appendix. Given the data at the beginning of this subsection, we can take two pullbacks:
\diag{ T \ar[d] & & S \ar[ll] \ar[d] \\
Y \times_X Y & & Y \times_X PX \times_X Y \ar[r]^(.6){p_2} \ar[d]_{(p_1,p_3)} \ar[ll]^(.55){(\Gamma(p_1,p_2),p_3)} & PX \ar[d]^{(s, t)} \\
& & Y \times Y \ar[r]_{f \times f} & X \times X.}
To get a better handle on $S$ it will be helpful to make use of the language of generalised elements. Indeed, from the universal property of $S$ it follows that there is a one-to-one correspondence between maps $Z \to S$ and quadruples $(y_0, y_1, \alpha, t)$, where $y_0, y_1: Z \to Y$ are two maps, $\alpha: Z \to PX$ is such that $s\alpha = f(y_0), t\alpha = f(y_1)$ and $t: Z \to T$ satisfies $\tau t = (\Gamma(y_0, \alpha), y_1)$. This justifies the idea of thinking of $S$ as a ``set'' with elements of the form $(y_0 \in Y, y_1 \in Y, \alpha \in PX, t \in T)$; we will use such set-theoretic language below and trust that the reader can easily translate arguments in this language into diagrammatic proofs, if he or she wishes.

\begin{theo}{mainconstrforTconnections}
The object $S \to Y \times Y$ defined above is an equivalence relation.
\end{theo}
\begin{proof}
We will use the language of generalised elements. We check:
\begin{enumerate}
\item Since $\Gamma$ is a $T$-transport, there is for any $y \in Y$ an element $t(y)$ such that $\tau t(y) = (\Gamma(y,rfy), y)$. Therefore we can define a map $Y \to S$ by sending $y \in Y$ to $(y, y, rf(y), t(y))$, showing reflexivity.
\item To show symmetry, suppose $\alpha$ is a path in $PX$ and $\Gamma(y_0, \alpha) \sim_T y_1$. From the groupoid structure on $PX$ we obtain an element $\sigma \alpha \in PX$ such that $\mu(\alpha, \sigma \alpha) \simeq_{X \times X} rf(y_0)$. The previous lemmas imply that
\[ \Gamma(y_1, \sigma \alpha) \sim_T \Gamma(\Gamma(y_0, \alpha), \sigma \alpha)  \sim_T \Gamma(y_0, \mu(\alpha, \sigma \alpha)) \sim_T \Gamma(y_0,rfy_0) \sim_T y_0, \]
and hence there is also an element $(y_1,y_0,\sigma \alpha, t_1) \in S$ for some suitable $t_1$. This proves symmetry of $S$.
\item To prove transivity, suppose $\Gamma(y_0, \alpha) \sim_T y_1$ and $\Gamma(y_1, \beta) \sim_T y_2$. Then
\[ \Gamma(y_0,\mu(\alpha, \beta))  \sim_T \Gamma(\Gamma(y_0, \alpha), \beta)  \sim_T \Gamma(y_1, \beta) \sim_T y_2, \]
and hence $S$ is transitive.
\end{enumerate}
\end{proof}

\begin{prop}{dingetje}
Suppose $\tau': T' \to Y \times_X Y$ is an equivalence relation similar to $\tau$, and $\Gamma': Y \times_X PX \to Y$ is a $T'$-transport, and let $S' \to Y \times Y$ be the equivalence relation built from $T'$ in the same way as $S$ is built from $T$. Then $S$ and $S'$ are similar.
\end{prop}
\begin{proof}
By symmetry it suffices to construct a map $S' \to S$ over $Y \times Y$. To build it, we use the language of generalised elements. So let $(y_0, y_1, \alpha, t') \in S'$ be arbitrary, meaning that $s \alpha =f(y_0)$, $t\alpha = f(y_1)$ and $\tau' t' = ((\Gamma'(y_0, \alpha),y_1)$. Since $T$ and $T'$ are similar, there is a map $k: T' \to T$ such that $\tau k = \tau'$, showing that $\Gamma'$ is not just a $T'$-transport, but a $T$-transport as well. So \reflemm{Tconnunique} implies that there is a map $l: Y \times_X PX \to T$ such that $(\Gamma,\Gamma') = \tau l$. So for $t'' := l(y_0,\alpha)$ we have $\tau t'' = (\Gamma(y_0,\alpha), \Gamma'(y_0,\alpha))$. Since $T$ is an equivalence relation we can use transitivity on $t''$ and $t'$ to construct an element $t \in T$ such that $\tau t = (\Gamma(y_0,\alpha), y_1)$. Therefore $(y_0,y_1,\alpha, t) \in S$, as desired.
\end{proof}

As in the main text we may define a type $A$ to be \emph{contractible} if both $A \to 1$ and $PA \to A \times A$ have sections. More generally, a display map $f: Y \to X$ is \emph{contractible} if both $f$ itself and $P_X(Y) \to Y \times_X Y$ have sections.

\begin{prop}{morehappy} Suppose that in the setting of the previous theorem the morphisms $f$ and $\tau$ have sections (so ``$f$ is $T$-contractible''), and $X$ is contractible. Then both $Y \to 1$ and $S \to Y \times Y$ have sections as well (hence ``$Y$ is $S$-contractible'').
\end{prop}
\begin{proof}
We again reason using generalised elements. Clearly, if $X$ has a global section and $f: Y \to X$ has a section, $Y$ has a global section as well. Any two elements $y_0, y_1 \in Y$ yield elements $f(y_0)$ and $f(y_1)$ in $X$. Since $X$ is contractible, there is path $\alpha \in PX$ with $s\alpha = f(y_0)$ and $t\alpha = f(y_1)$. Then $f\Gamma(y_0, \alpha) = t \alpha = f(y_1)$, so $\Gamma(y_0, \alpha)$ and $y_1$ are elements in $Y$ living in the same fibre over $X$. Since $\tau$ has a section, there is an element $t \in T$ with $\tau t = (\Gamma(y_0, \alpha), y_1)$. We conclude that $(y_0, y_1, \alpha, t) \in S$, and hence $Y$ is $S$-contractible.
\end{proof}

\bibliographystyle{plain} \bibliography{hSetoids}

\end{document}